\newif\ifsiamlegacy
\siamlegacytrue
\documentclass[final]{siamltex}

\usepackage{amsmath,amssymb,amsfonts,mathtools,bm}
\usepackage{graphicx}
\usepackage{booktabs}
\usepackage{array}
\usepackage{enumitem}
\usepackage{multirow}
\usepackage{threeparttable}
\usepackage{float}
\usepackage{algorithm}
\usepackage[noend]{algpseudocode}
\graphicspath{{figs/}}
\ifsiamlegacy\else
  \setcitestyle{numbers}
\fi

\providecommand{\sgn}{\operatorname{sign}}

\providecommand{\diag}{\operatorname{diag}}
\providecommand{\R}{\mathbb{R}}
\providecommand{\norm}[1]{\left\lVert #1 \right\rVert}

\providecommand{\eps}{\varepsilon}

\providecommand{\texorpdfstring}[2]{#1}

\title{Shifted Matrix-Sign Reflectors for Prescribed-Index Eigenspace Reflection}
\author{Jin Zhao\thanks{Academy of Multidisciplinary Sciences (AMS), Capital Normal University, Beijing 100048, China (\texttt{zjin@cnu.edu.cn}). J.Z. was supported in part by the Beijing Natural Science Foundation (Grant No. JR25003) and the National Key R\&D Program of China (Grant No. 2025YFA1018700).}}
\providecommand{\headers}[2]{\markboth{#1}{#2}}
\headers{Shifted matrix-sign reflectors}{J. Zhao}

\makeatletter
\@ifundefined{c@theorem}{}{%
  \let\c@algorithm\c@theorem
}
\@ifundefined{newsiamthm}{%
  \@ifundefined{discussion}{\newtheorem{discussion}[theorem]{Discussion}}{}%
}{%
  \@ifundefined{proposition}{\newsiamthm{proposition}{Proposition}}{}%
  \@ifundefined{lemma}{\newsiamthm{lemma}{Lemma}}{}%
  \@ifundefined{corollary}{\newsiamthm{corollary}{Corollary}}{}%
  \@ifundefined{discussion}{\newsiamremark{discussion}{Discussion}}{}%
}%
\@ifundefined{MSCcodes}{\newenvironment{MSCcodes}{\begin{AMS}}{\end{AMS}}}{}
\ifsiamlegacy
  \newcommand{\siambibliographystyle}{siam}
\else
  \newcommand{\siambibliographystyle}{siamplain}
\fi
\makeatother

\begin{document}
\maketitle

\begin{abstract}
Spectral projectors and the reflectors derived from them are basic objects in
numerical linear algebra.  This paper studies the prescribed-index reflector
\(I-2P_k\), where \(P_k\) is the spectral projector associated with the first
\(k\) eigenvectors of a symmetric matrix.  If a shift \(s\) lies in the target
spectral gap, then this reflector is exactly the shifted matrix sign
\(\sgn(H-sI)\).  The exact identity is elementary, but its algorithmic
consequences are not: all admissible shifts give the same exact reflector, while
finite-step sign filters can have very different errors.  We analyze odd
sign-preserving spectral filters, prove local inheritance and discrete stability
for the induced reflector iterations, derive a gap-dependent Newton--Schulz
operator bound, and give deterministic admissibility certificates for inexact and
reused shifts.  The analysis identifies the shifted spectral margin as the
quantity controlling finite-step reflector accuracy and explains why the midpoint
shift is the natural default.  Numerical experiments separate the matrix-function
issues from the outer saddle-search dynamics: controlled spectra verify the
margin predictions, low-dimensional tests distinguish shifted signs from raw
signs, target-index scans probe non-small \(k\), and Allen--Cahn and dense timing
tests identify the regimes in which full-matrix sign filters are useful and the
stiff regimes in which stronger sign engines are needed.
\end{abstract}

\begin{keywords}
matrix sign function, matrix-function approximation, Newton--Schulz iteration,
sign-preserving filters, prescribed-index reflectors, saddle dynamics
\end{keywords}

\begin{MSCcodes}
65F60, 65F15, 65F30, 65K05
\end{MSCcodes}

\section{Introduction}

Let \(H\in\R^{d\times d}\) be symmetric with ordered eigenvalues
\(\lambda_1\le\cdots\le\lambda_d\), and let \(P_k\) be the spectral projector
onto the invariant subspace associated with the first \(k\) eigenvectors.  The
operator
\begin{equation}\label{eq:Rk_intro}
R_k=I-2P_k
\end{equation}
flips this prescribed \(k\)-dimensional subspace and leaves its orthogonal
complement unchanged.  Such reflectors are a natural numerical-linear-algebra
object: they are spectral functions of \(H\), they encode a selected invariant
subspace without explicitly storing a basis for its complement, and their action
is needed whenever one wants to reverse a prescribed part of the spectrum.

A motivating example comes from prescribed-index saddle search.  For an energy
\(E\) with explicit symmetric Hessian \(H(x)=\nabla^2E(x)\), high-index saddle
dynamics and related methods use \(R_k(x)=I-2P_k(x)\) to reverse the unstable
eigenspace associated with a target Morse index \(k\), while keeping descent on
the complementary subspace \cite{pHiSD,AHiSD,iHiSD,GHiSD,HiSD,HiSDEuler}.  This
view complements dimer-type, gentlest-ascent, and variational or minimax
approaches to saddle computation \cite{ZhangDu2012ShrinkDimer,OSD,GGAD2025,GAD,
GaoLengZhou2015IMF,LevittOrtner2017Walker,LiZhou2001Minimax}.  In matrix-free
settings, evolving a \(k\)-frame or a projector is often the natural route.  This
paper instead studies the complementary explicit-Hessian regime, where \(H\) is
available as a dense or block-dense symmetric matrix and the central task is the
construction or approximation of the reflector itself.

The starting point is the shifted sign identity
\begin{equation}\label{eq:introexact}
I-2P_k=\sgn(H-sI),
\qquad \lambda_k<s<\lambda_{k+1}.
\end{equation}
At the exact spectral level, \eqref{eq:introexact} is immediate.  Its value is
algorithmic: it turns prescribed-index reflection into a shifted matrix-function
problem.  Matrix sign functions are classical tools for invariant-subspace
computation and related divide-and-conquer procedures
\cite{BaiDemmel1998SignInvariant,ByersHeMehrmann1997SignInvariant,HighamBook,
KenneyLaubSign}.  Here the objective is not to compute a general invariant
subspace, but to build or approximate the reflector \(I-2P_k\) used by an outer
index-selective iteration.  This distinction makes the shift important.  Every
shift inside the target gap gives the same exact reflector, but a finite-step
matrix-function approximation sees a shifted spectrum whose distance from the
singularity at zero can vary substantially.

The paper addresses three questions that arise from this reformulation.  First,
which scalar filters preserve the local index-selective structure of the exact
reflector?  Second, how does the shift affect finite-step sign approximation and
reflector-direction error?  Third, in which computational regimes is a
full-matrix sign realization a reasonable alternative to subspace tracking?  The
answers are intentionally local and spectral.  We do not claim that full-matrix
sign filters replace low-rank eigenspace tracking when \(k\) is small or only
Hessian-vector products are available.  Rather, the goal is to characterize a
complementary dense-kernel regime and to provide the margin and stability
analysis needed to use shifted sign filters reliably in that regime.

The exact identity \(I-2P_k=\operatorname{sign}(H-sI)\), valid for
\(s\in(\lambda_k,\lambda_{k+1})\), is the starting point of the paper.
Its algorithmic significance is that the exact reflector is independent of
the admissible shift, whereas finite-step matrix-function realizations are
strongly affected by the shifted spectral margin.  The main analytical
contributions are as follows.
\begin{enumerate}[leftmargin=1.8em]
\item We recast prescribed-index reflector construction as a shifted
      matrix-sign approximation problem and prove that odd sign-preserving
      filters preserve the local index-selective reflector geometry.  We also
      quantify reflector-direction errors through spectral sign-approximation
      errors and establish the corresponding local discrete stability result.

\item For finite-step Newton--Schulz reflectors, we derive a gap-dependent
      operator error bound on spectrally separated intervals and explain how
      conservative scaling compresses the effective shifted spectral margin,
      especially for stiff spectra arising from semidiscrete problems.

\item We develop shift-placement and admissibility theory for the shifted-sign
      realization: midpoint shifting is shown to maximize the worst-case
      shifted spectral margin, and deterministic certificates are given for
      inexact and reused shifts.
\end{enumerate}
These results are complemented by numerical experiments that separate the
matrix-function effects from the outer saddle-search dynamics.  The experiments
compare exact spectral reflectors, tracked-subspace baselines, raw signs, and
shifted finite-step sign filters in controlled spectral tests, target-index
scans, stiff semidiscrete examples, and dense timing tests.

The rest of the paper is organized as follows.  Section~2 establishes the exact
identity and its local index-selective consequences.  Section~3 studies
sign-preserving filters and specializes the quantitative analysis to
Newton--Schulz.  Section~4 gives the shift, admissibility, algorithmic, and
local discrete stability analysis.  Section~5 presents the numerical evidence,
and Section~6 concludes.

\section{Exact reflector identity and local dynamics}

\subsection{The ideal reflector as a shifted matrix sign}

Let $H\in\R^{d\times d}$ be symmetric, with eigen-decomposition
\[
H=Q\Lambda Q^T,
\qquad
\Lambda=\diag(\lambda_1,\ldots,\lambda_d),
\qquad
\lambda_1\le\cdots\le\lambda_d.
\]
Fix a target index $k\in\{1,\ldots,d-1\}$, write
$Q_k=[q_1,\ldots,q_k]\in\R^{d\times k}$, and define the orthogonal projector
\begin{equation}\label{eq:Pk}
P_k = Q_kQ_k^T = \sum_{i=1}^k q_iq_i^T.
\end{equation}

\begin{proposition}[exact reflector identity]\label{prop:exact}
If $s\in(\lambda_k,\lambda_{k+1})$, then
\begin{equation}\label{eq:exact}
\sgn(H-sI)=I-2P_k.
\end{equation}
\end{proposition}

\begin{proof}
Because $H-sI=Q\diag(\lambda_1-s,\ldots,\lambda_d-s)Q^T$ and
$s\in(\lambda_k,\lambda_{k+1})$, the first $k$ shifted eigenvalues are negative
and the remaining $d-k$ are positive.  Hence
\[
\sgn(H-sI)
=Q\diag(\underbrace{-1,\ldots,-1}_{k},
        \underbrace{1,\ldots,1}_{d-k})Q^T
 =\sum_{i=k+1}^d q_iq_i^T-\sum_{i=1}^k q_iq_i^T
 =I-2P_k.
\]
\end{proof}

Proposition~\ref{prop:exact} is the algebraic entry point.  Once the shift is
placed in the target spectral gap, the prescribed-index reflector is a matrix
sign.  All approximation issues are therefore shifted to the choice of sign
engine, scaling, and gap placement.

\subsection{Relation to subspace-tracking realizations}

Identity \eqref{eq:exact} reorganizes the reflector evaluation step.  A full
HiSD method evolves both the position variable and a frame or projector that
tracks the target unstable eigenspace \cite{pHiSD,GHiSD,HiSDEuler}.  In that
language,
\begin{itemize}[leftmargin=1.6em]
\item in standard HiSD, $R_k(x)$ is built from an explicitly evolved frame or
      projector;
\item in a sign realization, $R_k(x)$ is built from
      $\sgn(H(x)-s(x)I)$ or from an approximation thereof.
\end{itemize}
For explicit Hessians, the second route replaces repeated orthonormalization in
the $x$-equation by matrix-function evaluation together with gap estimation.
The resulting numerical kernels are therefore different, even though the target
reflector geometry is the same.

\subsection{Local prescribed-index selectivity}

We now record the local dynamical meaning of
Proposition~\ref{prop:exact}.  Consider the shifted-sign flow
\begin{equation}\label{eq:shiftedflow}
\dot x = -\sgn\!\bigl(H(x)-s_\star I\bigr)\nabla E(x),
\end{equation}
where $s_\star$ is a fixed shift placed inside the target gap of the critical
point under study.

\begin{proposition}[local selectivity]\label{prop:local}
Let $x_\star$ be a nondegenerate critical point of $E$ with Hessian
$H_\star=\nabla^2E(x_\star)$ and Morse index $j$.
Fix a target index $k\in\{1,\ldots,d-1\}$ and choose
$s_\star\in(\lambda_k(H_\star),\lambda_{k+1}(H_\star))$.
Then the linearization of \eqref{eq:shiftedflow} at $x_\star$ is
\[
\dot \xi = -\sgn(H_\star-s_\star I)H_\star \,\xi,
\]
and has negative spectrum if and only if $j=k$.
More precisely, the number of positive eigenvalues of the linearization equals
$|j-k|$.
\end{proposition}

\begin{proof}
Write $H_\star=Q\diag(\lambda_1,\ldots,\lambda_d)Q^T$.
Since $x_\star$ is nondegenerate,
$\nabla E(x_\star+\xi)=H_\star\xi+O(\norm{\xi}^2)$.
Therefore the linearization is
\[
-Q\diag\!\bigl(\sgn(\lambda_i-s_\star)\lambda_i\bigr)_{i=1}^dQ^T.
\]
If $j=k$, then $\lambda_i<0$ for $i\le k$ and $\lambda_i>0$ for $i>k$, while
$\lambda_i-s_\star<0$ for $i\le k$ and $\lambda_i-s_\star>0$ for $i>k$.
Hence every diagonal entry equals either $\lambda_i<0$ or $-\lambda_i<0$.
If $j<k$, then for $i=j+1,\ldots,k$ one has
$\lambda_i>0$ and $\lambda_i-s_\star<0$, which produces $k-j$ positive
eigenvalues.
If $j>k$, then for $i=k+1,\ldots,j$ one has
$\lambda_i<0$ and $\lambda_i-s_\star>0$, which produces $j-k$ positive
eigenvalues.
\end{proof}

Proposition~\ref{prop:local} gives the precise local fixed-index behavior one
expects from an ideal prescribed-index reflector.  This is the local statement
that the rest of the paper aims to preserve under approximate sign evaluation.

\subsection{Why the raw \texorpdfstring{$\sgn(H)$}{sign(H)} flow is not a prescribed-index method}

The unshifted flow
\begin{equation}\label{eq:rawflow}
\dot x = -\sgn(H(x))\nabla E(x)
\end{equation}
is superficially attractive because it uses no target shift.  Its local dynamics
is different in a decisive way.

\begin{proposition}\label{prop:raw}
Let $x_\star$ be a nondegenerate critical point with Hessian $H_\star$.
The linearization of \eqref{eq:rawflow} at $x_\star$ is
\[
\dot \xi = -\sgn(H_\star)H_\star\,\xi = -|H_\star|\,\xi.
\]
Consequently every nondegenerate critical point of $E$ is locally exponentially
stable for the raw sign flow.
\end{proposition}

\begin{proof}
If $H_\star=Q\diag(\lambda_1,\ldots,\lambda_d)Q^T$, then
\[
\sgn(H_\star)H_\star
=Q\diag(\sgn(\lambda_i)\lambda_i)Q^T
=Q\diag(|\lambda_i|)Q^T.
\]
Hence the linearization equals $-Q\diag(|\lambda_i|)Q^T$, which is negative
definite.
\end{proof}

Proposition~\ref{prop:raw} explains the role of raw $\sgn(H)$ in this paper.
It is a useful \emph{adaptive-index stationary-point search baseline}; it is not
a prescribed-index realization of the HiSD reflector.

\section{Approximate reflectors via sign-preserving filters}

The exact identity \eqref{eq:exact} suggests a broader class of reflector
surrogates.  Rather than computing the exact matrix sign, we apply a scalar
filter to the shifted Hessian spectrum.

\subsection{A general sign-preserving theorem}

Let $I\subset\R$ be an interval and let $\phi:I\to\R$ be a continuous odd
function such that
\begin{equation}\label{eq:signpreserving}
\phi(t)\,t>0\qquad\text{for every }t\in I\setminus\{0\}.
\end{equation}
We call such a $\phi$ a \emph{sign-preserving filter} on $I$.
For a scaling parameter $\alpha>0$ and a shift $s$, define the filtered
reflector
\begin{equation}\label{eq:phi-reflector}
\mathcal R_{\phi,\alpha}(H;s)=\phi\!\bigl(\alpha(H-sI)\bigr)
\end{equation}
via the standard spectral functional calculus.

\begin{theorem}[local selectivity for filtered reflectors]\label{thm:generic}
Let $x_\star$ be a nondegenerate critical point of $E$ with Hessian $H_\star$
and Morse index $j$.  Fix a target index $k$, choose a shift
$s_\star\in(\lambda_k(H_\star),\lambda_{k+1}(H_\star))$, and choose
$\alpha>0$ so that every scaled shifted eigenvalue
$\alpha(\lambda_i(H_\star)-s_\star)$ lies in $I$.
Consider the filtered dynamics
\begin{equation}\label{eq:filteredflow}
\dot x = -\mathcal R_{\phi,\alpha}(H(x);s_\star)\nabla E(x).
\end{equation}
Then the linearization of \eqref{eq:filteredflow} at $x_\star$ has exactly the
same inertia as the exact shifted-sign linearization in
Proposition~\ref{prop:local}.  In particular, $x_\star$ is locally
exponentially stable if and only if $j=k$.
\end{theorem}

\begin{proof}
Because $\mathcal R_{\phi,\alpha}(H_\star;s_\star)$ is a spectral function of
$H_\star$, it is diagonalized by the eigenbasis of $H_\star$:
\[
\mathcal R_{\phi,\alpha}(H_\star;s_\star)
=Q\diag\!\bigl(\phi(\alpha(\lambda_i-s_\star))\bigr)Q^T.
\]
As in the proof of Proposition~\ref{prop:local}, the linearization equals
\[
-Q\diag\!\bigl(\phi(\alpha(\lambda_i-s_\star))\lambda_i\bigr)Q^T.
\]
By \eqref{eq:signpreserving},
$\sgn(\phi(\alpha(\lambda_i-s_\star)))=\sgn(\lambda_i-s_\star)$,
so the sign of every diagonal entry is identical to the exact shifted-sign case.
The same case split as in Proposition~\ref{prop:local} therefore yields the
result.
\end{proof}

Theorem~\ref{thm:generic} is the main local approximation theorem of the paper.
Its content is simple but useful: local prescribed-index behavior depends on a
scalar sign-preservation property, not on one specific iteration.

\subsection{Operator and direction error}

Theorem~\ref{thm:generic} is qualitative.  For quantitative control we separate
the exact reflector from the filter approximation.

\begin{proposition}[spectral error formula]\label{prop:spectralerr}
Let $A\in\R^{d\times d}$ be symmetric with eigenvalues $\mu_1,\ldots,\mu_d$,
let $R=\sgn(A)$, and let $\widetilde R=\phi(A)$ for a scalar function
$\phi$ defined on the spectrum of $A$.  Then
\begin{equation}\label{eq:spectralerr}
\norm{\widetilde R-R}_2
=\max_{1\le i\le d}\bigl|\phi(\mu_i)-\sgn(\mu_i)\bigr|.
\end{equation}
\end{proposition}

\begin{proof}
Diagonalize $A=Q\diag(\mu_i)Q^T$.  Then
\[
\widetilde R-R
=Q\diag\!\bigl(\phi(\mu_i)-\sgn(\mu_i)\bigr)Q^T.
\]
For a symmetric matrix, the spectral norm equals the maximum absolute eigenvalue,
which gives \eqref{eq:spectralerr}.
\end{proof}

\begin{proposition}[direction error]\label{prop:direrr}
Let $A=H-sI$ be symmetric, set $R=\sgn(A)$, and let $\widetilde R$ be any
symmetric reflector surrogate.  For the exact and approximate search directions
$d=-Rg$ and $\widetilde d=-\widetilde Rg$,
\begin{equation}\label{eq:direrr}
\norm{\widetilde d-d}_2\le \norm{\widetilde R-R}_2\,\norm{g}_2.
\end{equation}
\end{proposition}

\begin{proof}
This is immediate from $\widetilde d-d=-(\widetilde R-R)g$.
\end{proof}

Propositions~\ref{prop:spectralerr}--\ref{prop:direrr} explain the numerical
role of the shift.  The shift does not change the exact reflector as long as it
remains inside the target gap; it changes the \emph{scalar approximation error}
of the finite-step filter, which then propagates directly to the reflector and
hence to the search direction.

\subsection{Newton--Schulz as a principal example}

For a symmetric matrix $A$ with no zero eigenvalues, the scaled Newton--Schulz
iteration reads
\begin{equation}\label{eq:NS}
X_{m+1}=\tfrac12 X_m(3I-X_m^2),\qquad X_0=\alpha A.
\end{equation}
The corresponding scalar polynomials are
\begin{equation}\label{eq:pm}
p_0(t)=t,\qquad
p_{m+1}(t)=\tfrac12 p_m(t)\bigl(3-p_m(t)^2\bigr).
\end{equation}

\begin{lemma}\label{lem:NSsign}
If $t\in[-1,1]\setminus\{0\}$, then
\[
\sgn\bigl(p_m(t)\bigr)=\sgn(t)\qquad\text{for every }m\ge 0.
\]
\end{lemma}

\begin{proof}
Let $f(t)=\frac12 t(3-t^2)$.
If $t\in(0,1]$, then $f(t)>0$ and
\[
f(t)-1=-\tfrac12(t-1)^2(t+2)\le 0,
\]
so $f(t)\in(0,1]$.
By oddness, $f(t)\in[-1,0)$ whenever $t\in[-1,0)$.
Since $p_{m+1}=f\circ p_m$ and $p_0(t)=t$, induction gives the claim.
\end{proof}

\begin{corollary}\label{cor:NS}
Assume the setting of Theorem~\ref{thm:generic} and choose
$\alpha\le \rho(H_\star-s_\star I)^{-1}$.
Then for every $m\ge 0$, the finite-step Newton--Schulz reflector
\[
\mathcal R_m(H;s_\star)=p_m\!\bigl(\alpha(H-s_\star I)\bigr)
\]
defines a flow of the form \eqref{eq:filteredflow} that is locally
exponentially stable at $x_\star$ if and only if $x_\star$ has Morse index $k$.
\end{corollary}

\begin{proof}
The scaling assumption places the eigenvalues of
$\alpha(H_\star-s_\star I)$ in $[-1,1]$.
Lemma~\ref{lem:NSsign} shows that $\phi=p_m$ satisfies the sign-preservation
hypothesis \eqref{eq:signpreserving}, so Theorem~\ref{thm:generic} applies.
\end{proof}

Corollary~\ref{cor:NS} isolates the reason Newton--Schulz is viable here:
finite-step Newton--Schulz can already preserve the correct local fixed-index
geometry, even when the exact sign is not yet accurately approximated.  The
latter issue is quantitative and is governed by
Propositions~\ref{prop:spectralerr}--\ref{prop:direrr}.  We emphasize again that
Newton--Schulz is only a principal example; the same framework applies to any
sign engine that yields an odd sign-preserving scalar filter on the relevant
interval.

A simple quantitative bound is also available on spectrally gapped intervals.

\begin{proposition}[quantitative Newton--Schulz error on a gapped interval]
\label{prop:NSerr}
Let \(p_m\) be defined by \eqref{eq:pm}, and let \(\gamma\in(0,1]\).
For \(t\in[\gamma,1]\), define the scalar sign error
\[
\varepsilon_m(t)=1-p_m(t).
\]
Then \(p_m\) maps \([0,1]\) into itself, is increasing on \([0,1]\), and
\begin{equation}\label{eq:ns_scalar_recurrence}
\varepsilon_{m+1}(t)=\tfrac12\varepsilon_m(t)^2\bigl(3-\varepsilon_m(t)\bigr)
\le \tfrac32\varepsilon_m(t)^2.
\end{equation}
Consequently, if a symmetric matrix \(A\) has spectrum contained in
\([\!-1,-\gamma]\cup[\gamma,1]\), then
\begin{equation}\label{eq:NSoperr}
\norm{p_m(A)-\sgn(A)}_2
\le \max_{|u|\in[\gamma,1]} |p_m(u)-\sgn(u)|
= \varepsilon_m(\gamma).
\end{equation}
\end{proposition}

\begin{proof}
Write \(f(t)=\tfrac12 t(3-t^2)\), so \(p_{m+1}=f\circ p_m\).  Since
\(f([0,1])\subset[0,1]\) and \(f'(t)=\tfrac32(1-t^2)\ge 0\) on \([0,1]\),
induction shows that every \(p_m\) maps \([0,1]\) into itself and is increasing.
Now let \(e=\varepsilon_m(t)=1-p_m(t)\in[0,1]\).  Substituting
\(p_m(t)=1-e\) into \eqref{eq:pm} gives
\[
\varepsilon_{m+1}(t)=1-\tfrac12(1-e)\bigl(3-(1-e)^2\bigr)
=\tfrac12 e^2(3-e),
\]
which yields \eqref{eq:ns_scalar_recurrence}.  Because \(p_m\) is odd, the
pointwise sign error \(|p_m(u)-\sgn(u)|\) depends only on \(|u|\) and is largest
at the smallest admissible magnitude \(\gamma\).  Applying the spectral error
formula \eqref{eq:spectralerr} therefore gives \eqref{eq:NSoperr}.
\end{proof}

Proposition~\ref{prop:NSerr} makes the shift-separation effect explicit for
finite-step Newton--Schulz.  After scaling the shifted Hessian into \([\!-1,1]\),
the relevant quantity is the smallest absolute shifted eigenvalue.  A larger
spectral margin means a smaller worst-case scalar sign error and hence a smaller
reflector error.  This is the quantitative perspective behind the midpoint
principle developed in the next section.

\begin{discussion}[spectral compression in stiff semidiscrete problems]
\label{disc:spectral_compression}
It is useful to isolate the approximation-theoretic mechanism behind the
Newton--Schulz depth requirement.  Let
\[
\gamma=\min_i |\alpha(\lambda_i(H)-s)|
\]
be the scaled spectral margin of the shifted Hessian.  Under the conservative
choice \(\alpha=\|H-sI\|_2^{-1}\), one has
\[
\gamma=
\frac{\min_i |\lambda_i(H)-s|}{\|H-sI\|_2}.
\]
Proposition~\ref{prop:NSerr} shows that the worst-case reflector error is
controlled by the scalar quantity \(\varepsilon_m(\gamma)=1-p_m(\gamma)\).
Equivalently, for a prescribed reflector tolerance \(\varepsilon\), the
smallest admissible Newton--Schulz depth is determined by the condition
\(\varepsilon_m(\gamma)\le \varepsilon\), which becomes rapidly more demanding
as \(\gamma\downarrow 0\).  For small positive \(t\),
\[
p_{m+1}(t)=\tfrac32 t + O(t^3),
\]
so Newton--Schulz first passes through a near-linear amplification regime before
it reaches the quadratic regime near \(1\).  Consequently, small values of
\(\gamma\) force substantially larger filter depths even when the sign pattern is
already correct.  In semidiscrete PDE settings, safe scalings typically satisfy
\(\alpha=O(\|H-sI\|_2^{-1})\); for second-order operators, \(\|H\|_2\) often
grows like \(h^{-2}\), so the scaled gap can be strongly compressed as the
discretization is refined.  This mechanism explains why shallow polynomial sign
filters deteriorate in the Allen--Cahn experiments of Section~5.
\end{discussion}

\subsection{Other admissible sign engines and why Newton--Schulz is the main baseline}

The sign-preserving viewpoint is deliberately broader than one iteration.
At one extreme, one may evaluate the shifted sign exactly from a full spectral
or Schur decomposition; this provides the highest-fidelity reflector baseline
but is usually the most expensive option in dense arithmetic.
At the other extreme, one may use low-degree odd polynomial filters that merely
preserve sign on the relevant shifted spectrum; these are analytically simple
but may require more outer iterations because the reflector error is larger.
Between these extremes lie rational and Newton-like sign iterations discussed in
standard matrix-function references
\cite{HighamBook,KenneyLaub1991PolarSignCond,KenneyLaubSign}.  The matrix sign
function is also a classical tool for invariant-subspace computation and related
divide-and-conquer strategies
\cite{BaiDemmel1998SignInvariant,ByersHeMehrmann1997SignInvariant}.  Their
appeal depends on which kernels are cheapest in the target environment.

We single out Newton--Schulz for the numerical part of the paper for a very
specific reason.  In the explicit-Hessian regime emphasized here, scaled
Newton--Schulz is inversion-free and dominated by dense matrix--matrix products,
so its computational profile aligns naturally with BLAS-3 and accelerator/GPU
implementations \cite{ChenChow}.  The paper therefore uses Newton--Schulz as the
principal sign engine, while keeping the theory at the level of generic
sign-preserving filters so that alternative realizations remain admissible.

\section{Shift placement, admissibility, and algorithmic realization}

The second design variable is the shift itself.  This section turns the midpoint
heuristic into a small collection of precise lemmas and explains how the theory
translates into an implementation.

\subsection{The midpoint shift}

\begin{proposition}[midpoint optimality]\label{prop:mid}
Let $\lambda_k<\lambda_{k+1}$ and define
\[
\delta(s)=\min\{s-\lambda_k,\lambda_{k+1}-s\},
\qquad s\in(\lambda_k,\lambda_{k+1}).
\]
Then $\delta(s)$ is uniquely maximized at
\begin{equation}\label{eq:midpoint}
s_\star = \frac{\lambda_k+\lambda_{k+1}}{2},
\end{equation}
with maximum value
\[
\delta(s_\star)=\frac{\lambda_{k+1}-\lambda_k}{2}.
\]
\end{proposition}

\begin{proof}
The function $\delta(s)$ is the minimum of one affine increasing function and
one affine decreasing function.
It is therefore maximized when the two are equal, namely when
$s-\lambda_k=\lambda_{k+1}-s$.
\end{proof}

Proposition~\ref{prop:mid} explains why the midpoint is the natural default.
The singular point of the sign function is zero; among all admissible shifts,
the midpoint pushes the shifted eigenvalues farthest from zero in the worst
case.  For finite-step sign filters, this is the cleanest local margin against
poor scalar approximation near the origin.

\begin{figure}[t]
\centering
\includegraphics[width=0.63\linewidth]{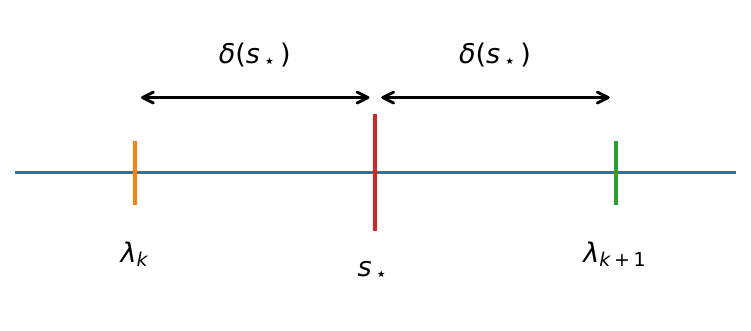}
\caption{For a fixed target gap $(\lambda_k,\lambda_{k+1})$, the midpoint shift
maximizes the minimum distance from $s$ to the two gap endpoints.  This is the
largest worst-case separation from the singular point of the sign function.}
\label{fig:midgap}
\end{figure}

\subsection{Inexact gap information and warm starts}

In practice, the gap endpoints are estimated rather than known exactly.
The following deterministic certificate records the two margin checks used by
the midpoint and reuse/refresh policies.

\begin{proposition}[inexact midpoint and reuse/refresh admissibility]
\label{prop:reuse_refresh}
At an iterate $x_n$, write
\[
\lambda_{i,n}=\lambda_i(H(x_n)),\qquad i=k,k+1,
\]
and suppose that endpoint estimates satisfy
\[
|\widehat\lambda_{k,n}-\lambda_{k,n}|\le \eps_n,\qquad
|\widehat\lambda_{k+1,n}-\lambda_{k+1,n}|\le \eps_n.
\]
For a candidate shift $s_n$, define its estimated margin
\[
\widehat m_n=
\min\{s_n-\widehat\lambda_{k,n},
       \widehat\lambda_{k+1,n}-s_n\}.
\]
If $\widehat m_n>\eps_n$, then $s_n$ is admissible at $x_n$, with true margin at
least $\widehat m_n-\eps_n$.  Moreover, if the next-step endpoint drift obeys
\[
\Delta_n :=
\max_{i=k,k+1}|\lambda_{i,n+1}-\lambda_{i,n}|
\le \delta_n < \widehat m_n-\eps_n,
\]
then the same shift $s_n$ remains admissible at $x_{n+1}$, with true margin at
least $\widehat m_n-\eps_n-\delta_n$.  In particular, the drift condition holds
whenever $\|H(x_{n+1})-H(x_n)\|_2\le \delta_n$, by Weyl's eigenvalue
perturbation bound.

Finally, if at $x_{n+1}$ the true gap
$g_{n+1}=\lambda_{k+1,n+1}-\lambda_{k,n+1}$ is positive and fresh endpoint
estimates have error at most $\eps_{n+1}<g_{n+1}/2$, then the estimated midpoint
\[
\widehat s_{n+1}
=\frac{\widehat\lambda_{k,n+1}+\widehat\lambda_{k+1,n+1}}{2}
\]
is admissible at $x_{n+1}$.
\end{proposition}

\begin{proof}
The first claim follows from
\[
s_n-\lambda_{k,n}
\ge s_n-\widehat\lambda_{k,n}-\eps_n
\ge \widehat m_n-\eps_n>0
\]
and
\[
\lambda_{k+1,n}-s_n
\ge \widehat\lambda_{k+1,n}-s_n-\eps_n
\ge \widehat m_n-\eps_n>0.
\]
If the endpoint drift is bounded by $\delta_n$, then the same inequalities at
$x_{n+1}$ lose at most another $\delta_n$, giving margin at least
$\widehat m_n-\eps_n-\delta_n$.  Weyl's bound gives
$|\lambda_i(H(x_{n+1}))-\lambda_i(H(x_n))|
\le \|H(x_{n+1})-H(x_n)\|_2$ for each $i$.  For the final statement, the
estimated midpoint has distance at least $(g_{n+1}-2\eps_{n+1})/2$ from each
true endpoint.
\end{proof}

Proposition~\ref{prop:reuse_refresh} gives the safety rule behind the
reuse/refresh policy: reuse a shift while the certified endpoint drift is
smaller than the available estimated margin, and refresh to a new midpoint once
that margin is no longer reliable but the target gap is still resolved.

\subsection{A formal discrete reflector algorithm}

We now state the discrete algorithm that corresponds to the filtered reflector
realization studied throughout the paper.

\begin{algorithm}[t]
\caption{Shifted-sign reflector iteration in the explicit-Hessian regime}
\label{alg:ssr}
\begin{algorithmic}[1]
\Require Target index $k$, initial iterate $x_0$, step sizes $\{\eta_n\}$,
         sign engine $\phi_m$, and shift policy.
\For{$n=0,1,2,\ldots$ until a stopping criterion is met}
  \State Form the current gradient and Hessian:
         \(g_n=\nabla E(x_n)\), \(H_n=\nabla^2E(x_n)\).
  \State Estimate the target-gap endpoints
         \(\widehat\lambda_k(H_n)\) and
         \(\widehat\lambda_{k+1}(H_n)\).
  \State Choose an admissible shift \(s_n\) by one of the following policies.
  \Statex \hspace{\algorithmicindent}\emph{midpoint:}
          \(s_n=(\widehat\lambda_k+\widehat\lambda_{k+1})/2\).
  \Statex \hspace{\algorithmicindent}\emph{damped midpoint:}
          blend the previous admissible shift with the new midpoint and clip
          back into the estimated gap.
  \Statex \hspace{\algorithmicindent}\emph{reuse/refresh:}
          keep the previous shift while its admissibility margin is above a
          prescribed threshold, and otherwise reset to the midpoint.
  \State Choose a scaling \(\alpha_n\) so that the relevant spectrum of
         \(\alpha_n(H_n-s_nI)\) lies inside the scalar design interval of the
         sign engine.
  \State Form the filtered reflector
         \(\widetilde R_n=\phi_m(\alpha_n(H_n-s_nI))\), for example by
         \(m\) Newton--Schulz steps, by an exact spectral sign, or by another
         odd sign-preserving filter.
  \State Update the state by
        \begin{equation}\label{eq:discrete_step}
        x_{n+1}=x_n-\eta_n\widetilde R_n g_n.
        \end{equation}
  \State Stop if \(\|g_{n+1}\|\) is below tolerance, the shift ceases to be
         admissible, or a maximum iteration budget is reached.
\EndFor
\end{algorithmic}
\end{algorithm}

Algorithm~\ref{alg:ssr} is deliberately modular.  The theory does not require a
particular sign engine; it only requires that the resulting scalar filter be odd
and sign-preserving on the relevant shifted spectrum.  Likewise, the shift
policy only needs to keep $s_n$ inside the target gap.  In this sense the
algorithmic object of the paper is the filtered reflector step
\eqref{eq:discrete_step}, not one special recurrence.

\paragraph{Recommended default realization}
In a practical dense explicit-Hessian implementation, a natural default is to
estimate \(\widehat\lambda_k\) and \(\widehat\lambda_{k+1}\) by a dense
eigensolver or by a partial eigensolver near the target gap, for example with
implicitly restarted Arnoldi/Lanczos or block preconditioned eigensolvers
\cite{DuerschShaoYangGu2018LOBPCG,Knyazev2001LOBPCG,Lehoucq2001IRAM,
LehoucqSorensenYang1998ARPACK}, choose the midpoint shift, take
\(\alpha_n=\|H_n-s_nI\|_2^{-1}\) when a spectral-norm estimate is available and
otherwise a cheaper surrogate such as
\(\|H_n-s_nI\|_\infty^{-1}\), apply \(m\)-step Newton--Schulz with moderate
depths such as \(m=2,4,6\), reuse the previous shift while the certified
endpoint drift remains below the available estimated margin in
Proposition~\ref{prop:reuse_refresh}, and refresh more aggressively when that
margin becomes small.  Keep a fixed outer step size unless the local bound in
Corollary~\ref{cor:discrete_delta} suggests a smaller value.  Section~5
instantiates the numerical tests in this spirit.

\begin{discussion}[scaling choice and filter efficiency]
The scaling $\alpha_n$ in Algorithm~\ref{alg:ssr} is a genuine algorithmic
design variable.  Choosing $\alpha_n$ close to $\|H_n-s_nI\|_2^{-1}$
(equivalently, to $\rho(H_n-s_nI)^{-1}$ in the symmetric setting) enlarges the
effective scaled margin
\[
\gamma_n=\min_i |\alpha_n(\lambda_i(H_n)-s_n)|
\]
and therefore improves the finite-step filter accuracy predicted by
Proposition~\ref{prop:NSerr}.  However, obtaining such a scaling may require
extra spectral estimation.  More conservative choices such as
$\alpha_n=\|H_n-s_nI\|_\infty^{-1}$ or
$\alpha_n=\|H_n-s_nI\|_1^{-1}$ are cheaper, but they reduce $\gamma_n$,
compress the scaled spectrum toward the origin, and can therefore increase the
Newton--Schulz depth required to achieve a given reflector tolerance.  The
scaling policy thus balances setup cost against sign-filter efficiency.
\end{discussion}

A natural comparison point is a simple HiSD-style baseline based on explicit
tracking of a $k$-dimensional unstable subspace.  Let
$Q_n\in\R^{d\times k}$ have orthonormal columns, so that $Q_nQ_n^T$ is the
orthogonal projector onto the current tracked subspace.  A basic projected
Rayleigh update with tracking step size $\tau_{\mathrm{trk}}>0$ is
\[
\widehat Q_{n+1}=Q_n-\tau_{\mathrm{trk}}(I-Q_nQ_n^T)H_nQ_n,
\qquad
Q_{n+1}=\operatorname{qf}(\widehat Q_{n+1}),
\]
where $\operatorname{qf}(\cdot)$ denotes the orthonormal factor in a QR-type
retraction.  The associated reflector for the next position update is then
taken as
\[
R_n^{\mathrm{trk}}=I-2Q_{n+1}Q_{n+1}^T,
\]
which flips the tracked $k$-dimensional subspace and leaves its orthogonal
complement unchanged.  The resulting $x$-update is the corresponding reflected
gradient step.  For $k=1$, the frame update reduces to the familiar
vector-tracking form, and the reflector becomes
$R_n^{\mathrm{trk}}=I-2v_{n+1}v_{n+1}^T$ after normalization.  This baseline
provides a transparent subspace-tracking comparison for the numerical section
and makes the algorithmic dependence on the target index $k$ explicit.

\subsection{A local discrete stability theorem}

The local analysis in Sections~2--3 is continuous-time.  We now connect it to
Algorithm~\ref{alg:ssr}.

\begin{theorem}[local discrete stability of the index-$k$ filtered reflector iteration]
\label{thm:discrete}
Let $x_\star$ be a nondegenerate critical point of $E$ with Hessian
$H_\star=\nabla^2E(x_\star)$ and Morse index $j$.
Fix a target index $k$, and let $s(x)$ and $\alpha(x)$ be continuous scalar
policies in a neighborhood of $x_\star$ with
\[
s_\star:=s(x_\star)\in(\lambda_k(H_\star),\lambda_{k+1}(H_\star)),
\qquad \alpha_\star:=\alpha(x_\star)>0.
\]
Assume that every scaled shifted eigenvalue
$\alpha_\star(\lambda_i(H_\star)-s_\star)$ lies in the scalar design interval of
an odd sign-preserving filter $\phi$, and that, for $x$ sufficiently near
$x_\star$, the spectrum of $\alpha(x)(\nabla^2E(x)-s(x)I)$ remains in that
interval.
Consider the discrete map
\begin{equation}\label{eq:discrete_map}
F_\eta(x)=x-\eta\,\phi\!\bigl(\alpha(x)(\nabla^2E(x)-s(x) I)\bigr)\nabla E(x).
\end{equation}
Then the Jacobian of $F_\eta$ at $x_\star$ is diagonalizable in an orthonormal
eigenbasis of $H_\star$ and has eigenvalues
\begin{equation}\label{eq:discrete_eval}
\mu_i(\eta)=1-\eta\,\phi\!\bigl(\alpha_\star(\lambda_i(H_\star)-s_\star)\bigr)
\lambda_i(H_\star),
\qquad i=1,\ldots,d.
\end{equation}
If $j=k$, then every factor in \eqref{eq:discrete_eval} is positive and
$x_\star$ is a locally linearly attracting fixed point whenever
\begin{equation}\label{eq:eta_condition}
0<\eta<\frac{2}{\max_i\phi\!\bigl(\alpha_\star(\lambda_i(H_\star)-s_\star)\bigr)
\lambda_i(H_\star)}.
\end{equation}
If $j\neq k$, then for every $\eta>0$ at least $|j-k|$ eigenvalues of
$DF_\eta(x_\star)$ satisfy $|\mu_i(\eta)|>1$, so the fixed point is locally
unstable.
\end{theorem}

\begin{proof}
Write
\[
R(x)=\phi\!\bigl(\alpha(x)(\nabla^2E(x)-s(x)I)\bigr).
\]
The continuity assumptions and the continuity of the spectral functional
calculus for symmetric matrices give $R(x)\to R_\star:=R(x_\star)$ as
$x\to x_\star$.  For $h\to0$,
\[
\nabla E(x_\star+h)=H_\star h+o(\|h\|),
\]
and hence
\[
R(x_\star+h)\nabla E(x_\star+h)
=R_\star H_\star h+o(\|h\|).
\]
Thus $F_\eta$ is differentiable at $x_\star$ and
\[
DF_\eta(x_\star)=I-\eta\,\phi\!\bigl(\alpha_\star(H_\star-s_\star I)\bigr)H_\star.
\]
Since $\phi\bigl(\alpha_\star(H_\star-s_\star I)\bigr)$ is a spectral function of
$H_\star$, the two matrices are simultaneously diagonalizable in an orthonormal
eigenbasis of $H_\star$, which yields \eqref{eq:discrete_eval}.  If $j=k$, then
the gap
$(\lambda_k(H_\star),\lambda_{k+1}(H_\star))$ straddles zero and the
sign-preserving property implies
\[
\phi\!\bigl(\alpha_\star(\lambda_i(H_\star)-s_\star)\bigr)\lambda_i(H_\star)>0
\qquad\text{for all }i.
\]
Condition \eqref{eq:eta_condition} then gives $|\mu_i(\eta)|<1$ for every $i$,
so the fixed point is locally linearly attracting.  If $j\neq k$, then the same
sign comparison used in Proposition~\ref{prop:local} shows that exactly $|j-k|$
of the products
$\phi\bigl(\alpha_\star(\lambda_i(H_\star)-s_\star)\bigr)\lambda_i(H_\star)$ are
negative.  For those indices,
$\mu_i(\eta)=1+\eta|\phi(\alpha_\star(\lambda_i(H_\star)-s_\star))\lambda_i(H_\star)|>1$,
which proves local instability.
\end{proof}

\begin{corollary}[a near-sign step-size window]\label{cor:discrete_delta}
Under the hypotheses of Theorem~\ref{thm:discrete}, assume in addition that
$j=k$ and define
\[
\delta_\star:=\max_i\Bigl|\phi\bigl(\alpha_\star(\lambda_i(H_\star)-s_\star)\bigr)
-\sgn\bigl(\lambda_i(H_\star)-s_\star\bigr)\Bigr|<1.
\]
Then, for every $i=1,\ldots,d$,
\[
(1-\delta_\star)|\lambda_i(H_\star)|
\le
\phi\bigl(\alpha_\star(\lambda_i(H_\star)-s_\star)\bigr)\lambda_i(H_\star)
\le
(1+\delta_\star)|\lambda_i(H_\star)|.
\]
Consequently, the local attraction condition \eqref{eq:eta_condition} is implied by
\[
0<\eta<\frac{2}{(1+\delta_\star)\|H_\star\|_2}.
\]
In particular, for filter families whose local spectral error $\delta_\star$
decays with the filter depth, the admissible discrete step-size window tends to
that of the exact reflector iteration.
\end{corollary}

\begin{proof}
At a target index-$k$ saddle one has
$\sgn(\lambda_i(H_\star)-s_\star)\lambda_i(H_\star)=|\lambda_i(H_\star)|$ for every
$i$.  The error bound defining $\delta_\star$ therefore implies the two-sided
estimate above.  The stated step-size condition is then stronger than
\eqref{eq:eta_condition}, hence sufficient by Theorem~\ref{thm:discrete}.
\end{proof}

\begin{proposition}[local contraction robustness under reflector error]
\label{prop:local_contraction}
Let $x_\star$ be a nondegenerate critical point of $E$ with Morse index $k$,
let $H_\star=\nabla^2E(x_\star)$, and choose an admissible shift
$s_\star\in(\lambda_k(H_\star),\lambda_{k+1}(H_\star))$.  Set
\[
R_\star=\sgn(H_\star-s_\star I).
\]
For any symmetric reflector surrogate $\widetilde R_\star$, define the exact
and approximate frozen maps
\[
F_\eta^{\mathrm{ex}}(x)=x-\eta R_\star \nabla E(x),
\qquad
F_\eta^{\mathrm{ap}}(x)=x-\eta \widetilde R_\star \nabla E(x).
\]
Then
\begin{equation}\label{eq:jacobian_perturb}
\norm{DF_\eta^{\mathrm{ap}}(x_\star)-DF_\eta^{\mathrm{ex}}(x_\star)}_2
\le \eta \norm{\widetilde R_\star-R_\star}_2 \norm{H_\star}_2.
\end{equation}
If, in addition,
\[
0<\eta<\frac{2}{\norm{H_\star}_2}
\]
and
\[
\eta \norm{\widetilde R_\star-R_\star}_2 \norm{H_\star}_2
<
\beta_\eta,
\qquad
\beta_\eta:=1-\max_i \bigl|1-\eta |\lambda_i(H_\star)|\bigr|,
\]
then $F_\eta^{\mathrm{ap}}$ is locally linearly attracting at $x_\star$.
\end{proposition}

\begin{proof}
Because $\nabla E(x_\star)=0$,
\[
DF_\eta^{\mathrm{ex}}(x_\star)=I-\eta R_\star H_\star,
\qquad
DF_\eta^{\mathrm{ap}}(x_\star)=I-\eta \widetilde R_\star H_\star.
\]
At a Morse-index-$k$ critical point, Proposition~\ref{prop:exact} gives
$R_\star H_\star=|H_\star|$, so
\[
DF_\eta^{\mathrm{ex}}(x_\star)=I-\eta |H_\star|.
\]
Subtracting the two Jacobians yields
\[
DF_\eta^{\mathrm{ap}}(x_\star)-DF_\eta^{\mathrm{ex}}(x_\star)
=-\eta (\widetilde R_\star-R_\star)H_\star,
\]
which implies \eqref{eq:jacobian_perturb}.  The step-size condition
$0<\eta<2/\norm{H_\star}_2$ ensures $\beta_\eta>0$ and
\[
\norm{DF_\eta^{\mathrm{ex}}(x_\star)}_2
=\max_i \bigl|1-\eta |\lambda_i(H_\star)|\bigr|
=1-\beta_\eta.
\]
Therefore
\[
\norm{DF_\eta^{\mathrm{ap}}(x_\star)}_2
\le
\norm{DF_\eta^{\mathrm{ex}}(x_\star)}_2
\;+\;
\norm{DF_\eta^{\mathrm{ap}}(x_\star)-DF_\eta^{\mathrm{ex}}(x_\star)}_2
<1.
\]
Hence the approximate frozen map is locally linearly attracting at $x_\star$.
\end{proof}

Theorem~\ref{thm:discrete} is the algorithm-level counterpart of the local flow
results.  It shows that the discrete filtered iteration---with a
state-dependent admissible shift and target index $k$---inherits the correct
prescribed-index local structure provided that the shift is admissible at the
target saddle, the filter preserves sign on the shifted spectrum, and the step
size is chosen below the local spectral threshold \eqref{eq:eta_condition}.

\begin{discussion}[inner--outer coupling and adaptive depth]
The local theory also clarifies how reflector accuracy interacts with the outer
state update.  Proposition~\ref{prop:local_contraction} shows that reflector
error perturbs the local contraction margin of the discrete map, while
Proposition~\ref{prop:direrr} yields
\[
\norm{\widetilde d-d}_2\le \norm{\widetilde R-R}_2\,\norm{\nabla E(x)}_2,
\]
so a relatively coarse inner sign evaluation can still be acceptable when the
outer iterate is far from stationarity and \(\|\nabla E(x)\|_2\) is large.
Near a target saddle, however, the discrete linearization in
Theorem~\ref{thm:discrete} shows that preserving the correct sign pattern is
the minimal requirement for local index selectivity, while accurate sign
approximation becomes increasingly important if one also wants to retain the
contraction factors of the exact reflector iteration.  This inner--outer
coupling explains why shallow filters can be useful in early outer iterations,
while deeper filters or stronger sign engines are needed in the final local
regime.  It also motivates adaptive-depth and warm-start strategies for future
work.
\end{discussion}

\subsection{Computational regime and complexity trade-offs}

The sign formulation is most attractive when explicit Hessians or dense Hessian
blocks are already available.  In that regime, finite-step Newton--Schulz is
dominated by dense matrix--matrix products, which align naturally with BLAS-3
and accelerator kernels \cite{ChenChow}.  By contrast, in matrix-free settings
where only Hessian--vector products are available, tracked-subspace and
Krylov-type eigenspace methods are typically the more natural computational
building blocks.

A basic complexity tension should be stated explicitly.  The ideal reflector
\(R_k=I-2P_k\) is a rank-\(k\) modification of the identity, so when \(k\) is
small there are clear asymptotic incentives to work directly with the target
unstable subspace rather than with a full matrix function.  In dense arithmetic,
one Newton--Schulz step has \(O(d^3)\) cost, whereas tracked \(k\)-frame updates
are dominated by Hessian--frame products of \(O(d^2k)\) type, together with
lower-order orthogonalization costs.  Likewise, block Krylov/Lanczos procedures
targeting \(k\) extremal modes are typically built from \(O(d^2k)\)-type dense
Hessian--frame products per outer iteration, although their total cost also
depends on restart and convergence behavior; representative implementations
include implicitly restarted Arnoldi methods and block preconditioned
conjugate-gradient eigensolvers
\cite{DuerschShaoYangGu2018LOBPCG,Knyazev2001LOBPCG,Lehoucq2001IRAM,
LehoucqSorensenYang1998ARPACK}.  The matrix-sign viewpoint becomes
competitive when dense matrix kernels are highly optimized, when batched or
accelerator execution favors high arithmetic intensity, or when \(k\) is not
very small relative to \(d\).  The discussion therefore involves four
practically distinct reflector realizations: the exact spectral reflector, a
tracked-subspace baseline, the adaptive-index raw sign, and shifted-sign
realizations based on finite-step sign filters.  Their numerical attractiveness
depends on which information is available---full Hessians, Hessian actions, or a
usable target gap---and on which kernels dominate the cost.  The matrix-sign
route is therefore not a universal replacement for low-rank subspace tracking,
but a dense-kernel alternative whose attractiveness increases when explicit
Hessians, nearby repeated Hessians, batched workloads, or non-small target
indices make full-matrix primitives competitive.

\section{Numerical experiments}

The experiments are designed to test the matrix-function conclusions behind the
reflector construction.  We compare four reflector realizations: the exact
spectral reflector \(I-2P_k\), a simple tracked-subspace baseline, the raw sign
\(\sgn(H)\), and shifted finite-step sign filters.  Among shifted filters, the
midpoint rule is the default shift policy.  The supplementary code bundle
contains the scripts and tabulated outputs used to generate all figures.

The tracked baseline evolves a \(k\)-frame \(Q_n\) by an orthogonalized projected
Rayleigh step and uses \(I-2Q_nQ_n^T\) in the state update.  It is included as a
transparent subspace-tracking comparator rather than as an optimized eigensolver
implementation.  The comparison is therefore most informative about regimes:
small-\(k\) or matrix-free problems favor subspace methods, while dense,
batched, or non-small-\(k\) explicit-Hessian problems can make full-matrix sign
kernels competitive.

\subsection{Controlled spectral tests}

We first isolate the scalar matrix-function mechanism.  A synthetic spectrum
with dimension \(d=256\), target index \(k=32\), and prescribed gap
\((\lambda_k,\lambda_{k+1})\) is used; because the matrices are symmetric,
Proposition~\ref{prop:spectralerr} reduces the reflector error to the maximum scalar
sign error on the shifted spectrum.  Figure~\ref{fig:controlled_spectral}
reports three tests.  The left panel moves the shift
\(s=\lambda_k+\theta(\lambda_{k+1}-\lambda_k)\) through the target gap.  The
operator error is smallest near the midpoint, in agreement with
Proposition~\ref{prop:mid}.  The middle panel increases the spectral radius while
keeping the unscaled target gap fixed; conservative scaling then compresses the
scaled margin \(\gamma\), and the finite-step sign error grows as predicted by
Proposition~\ref{prop:NSerr} and Discussion~\ref{disc:spectral_compression}.
The right panel shows the same effect in scalar form: smaller \(\gamma\) requires
larger Newton--Schulz depth before the quadratic regime becomes visible.

\begin{figure}[H]
\centering
\includegraphics[width=0.95\linewidth]{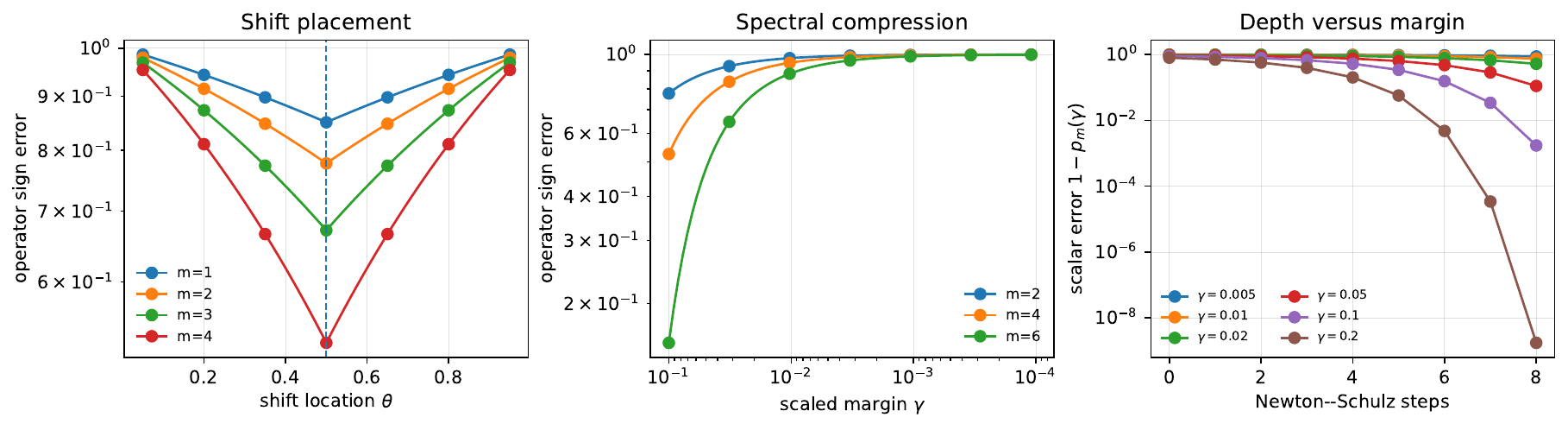}
\caption{Controlled spectral tests for finite-step shifted-sign filters.
Left: moving the shift through the target gap gives the smallest operator sign
error near the midpoint.  Middle: as conservative scaling compresses the scaled
margin \(\gamma\), fixed-depth Newton--Schulz errors increase.  Right: the scalar
error \(1-p_m(\gamma)\) decreases much more slowly when \(\gamma\) is small.}
\label{fig:controlled_spectral}
\end{figure}

\subsection{Raw sign versus prescribed-index reflection}

The M\"uller--Brown surface \cite{MullerBrown} provides a two-dimensional test
in which prescribed-index geometry is visible.  Figure~\ref{fig:muller} compares
trajectories from common starting points.  The exact reflector and the
midpoint-shifted sign realization follow closely aligned index-one reflector
trajectories.  The raw sign follows a different stationary-point-search geometry,
consistent with Proposition~\ref{prop:raw}: it stabilizes nondegenerate
critical points without selecting a prescribed Morse index.

\begin{figure}[t]
\centering
\includegraphics[width=0.95\linewidth]{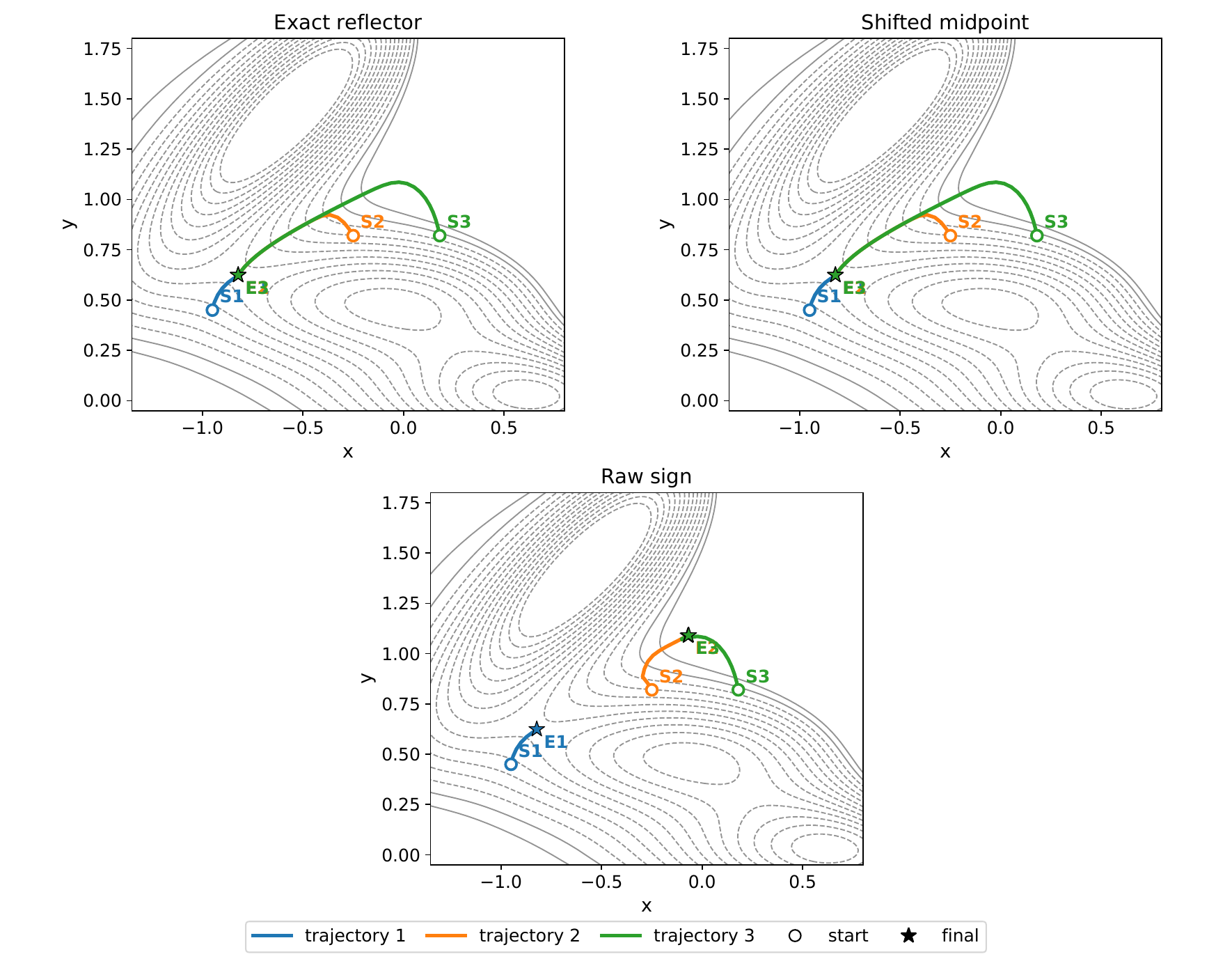}
\caption{Trajectory comparison on the M\"uller--Brown surface.
The exact reflector and midpoint-shifted sign realization follow closely aligned
prescribed-index trajectories, whereas raw \(\sgn(H)\) follows a different
stationary-point-search geometry.}
\label{fig:muller}
\end{figure}

\subsection{Target-index scan}

To test behavior beyond index one, we use a dense rotated quartic family with a
known index-\(k\) saddle at the origin.  Let \(y=Q^Tx\), where \(Q\) is a fixed
random orthogonal matrix, and set
\[
E_{k,d}(x)=\frac14\sum_{i=1}^d y_i^4+\frac12\sum_{i=1}^d c_i y_i^2,
\qquad
c_i<0\ (i\le k),\quad c_i>0\ (i>k).
\]
We use \(d=64\) and scan \(k\in\{1,2,4,8,16\}\) from matched random starts.
Figure~\ref{fig:k_scan} shows that a shallow midpoint filter with \(m=2\) is not
robust across the scan, while \(m=4\) and \(m=6\) recover the exact-reflector
success rate and sharply reduce local direction error.  In this dense
implementation, the tracked baseline becomes more costly as \(k\) grows.  The
experiment supports the regime statement rather than a universal dominance
claim: shifted full-matrix filters become more attractive when the target index
is not very small and the sign engine is sufficiently accurate.

\begin{figure}[t]
\centering
\includegraphics[width=0.95\linewidth]{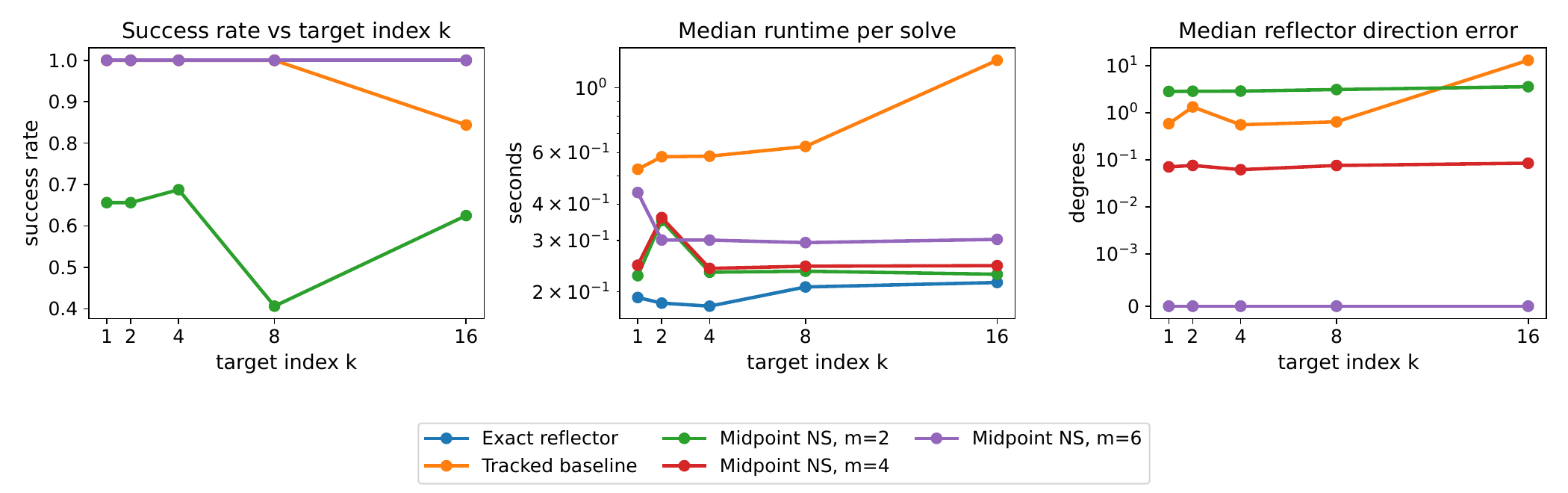}
\caption{Moderate-dimensional target-index scan on a dense rotated quartic
family with \(d=64\) and prescribed index \(k\in\{1,2,4,8,16\}\).
Left: success rate over 32 matched random starts.  Middle: median runtime per
solve.  Right: median one-step reflector direction error over local
perturbations.}
\label{fig:k_scan}
\end{figure}

\subsection{Stiff spectra: Allen--Cahn}

The one-dimensional semidiscrete Allen--Cahn energy with Neumann boundary
conditions is
\[
E_h(u)= h\sum_{i=1}^n \frac{\eps^2}{2}\left(\frac{u_{i+1}-u_i}{h}\right)^2
+ h\sum_{i=1}^n \frac{(u_i^2-1)^2}{4}.
\]
With the standard ghost-point Neumann Laplacian,
\[
\nabla E_h(u)= -\eps^2\Delta_h u + u^{\circ 3}-u,
\qquad
H_h(u)= -\eps^2\Delta_h + \diag(3u_i^2-1).
\]
For \(\eps=0.4\) and \(n=41\), the zero state is an index-one saddle.  The test
therefore probes a stiff but explicit Hessian regime in which the target index
is easy to verify.

Figure~\ref{fig:allen_cahn} confirms the spectral-compression picture.  Exact
reflection and the tracked-subspace baseline remain closest to the discrete
saddle.  Midpoint-shifted filters improve as the Newton--Schulz depth increases,
and their local direction error decreases substantially with depth, but shallow
polynomial filters do not match tracked or exact reflection in this stiff
semidiscrete setting.  This is the practical boundary predicted by
Discussion~\ref{disc:spectral_compression}: the shifted-sign geometry is correct,
but a compressed scaled margin requires deeper or stronger sign engines.

\begin{figure}[t]
\centering
\includegraphics[width=0.95\linewidth]{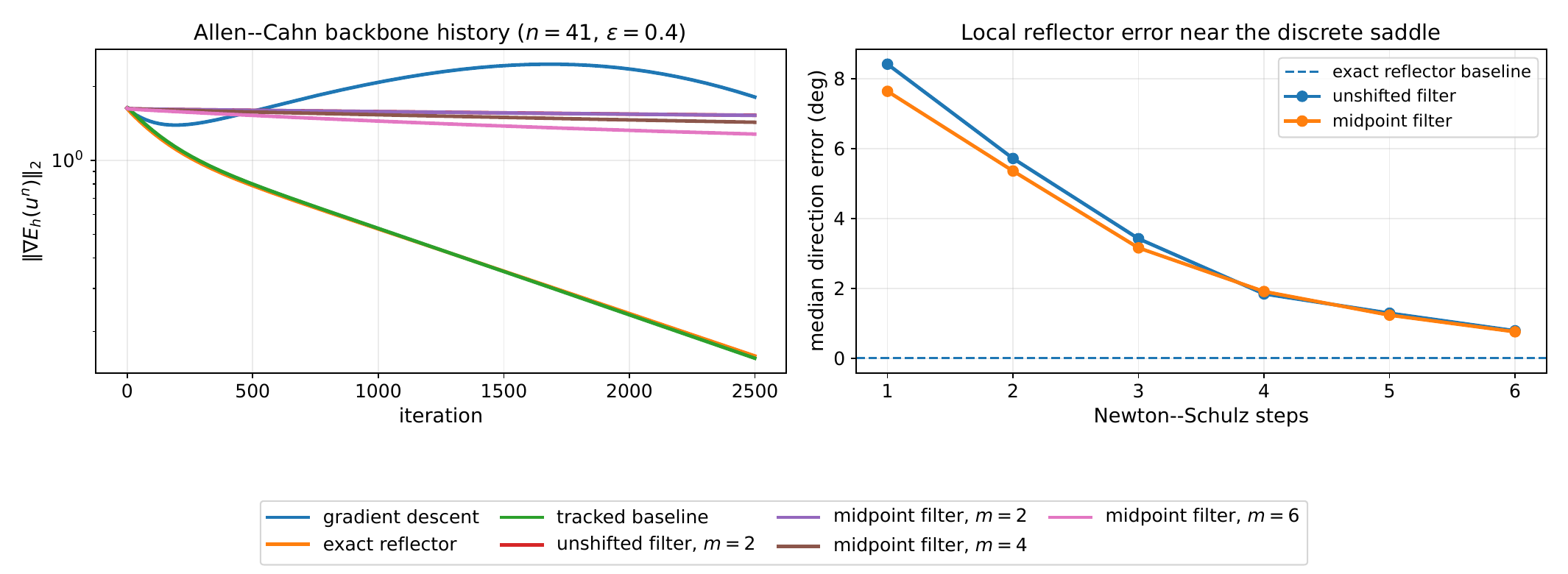}
\caption{Semidiscrete Allen--Cahn test.
Left: gradient-norm histories for gradient descent, the exact reflector, a
tracked-reflector baseline, an unshifted finite-step sign filter, and
midpoint-shifted sign filters with depths \(m=2,4,6\).  Right: median one-step
direction error near the discrete saddle.}
\label{fig:allen_cahn}
\end{figure}

\subsection{Dense explicit-Hessian timings}

The final experiment separates kernel cost from setup cost.  Figure~\ref{fig:cost}
reports dense explicit-Hessian timings at Newton--Schulz depth \(m=2\).  The left
panel charges only the online sign kernels, assuming the shift and scale have
already been supplied by the policy; in this conditional comparison, the
Newton--Schulz cores are much cheaper than a full dense eigensolve at the larger
tested sizes.  The right panel charges the midpoint filter for exact spectral
setup of the shift and scale.  With exact setup included, the total midpoint
cost is comparable to a full eigensolve in this moderate-size test.  Thus the
matrix-sign route is most attractive when shift/scale information is inexpensive,
reused, warm-started, or amortized across related Hessians.

\begin{figure}[t]
\centering
\includegraphics[width=0.95\linewidth]{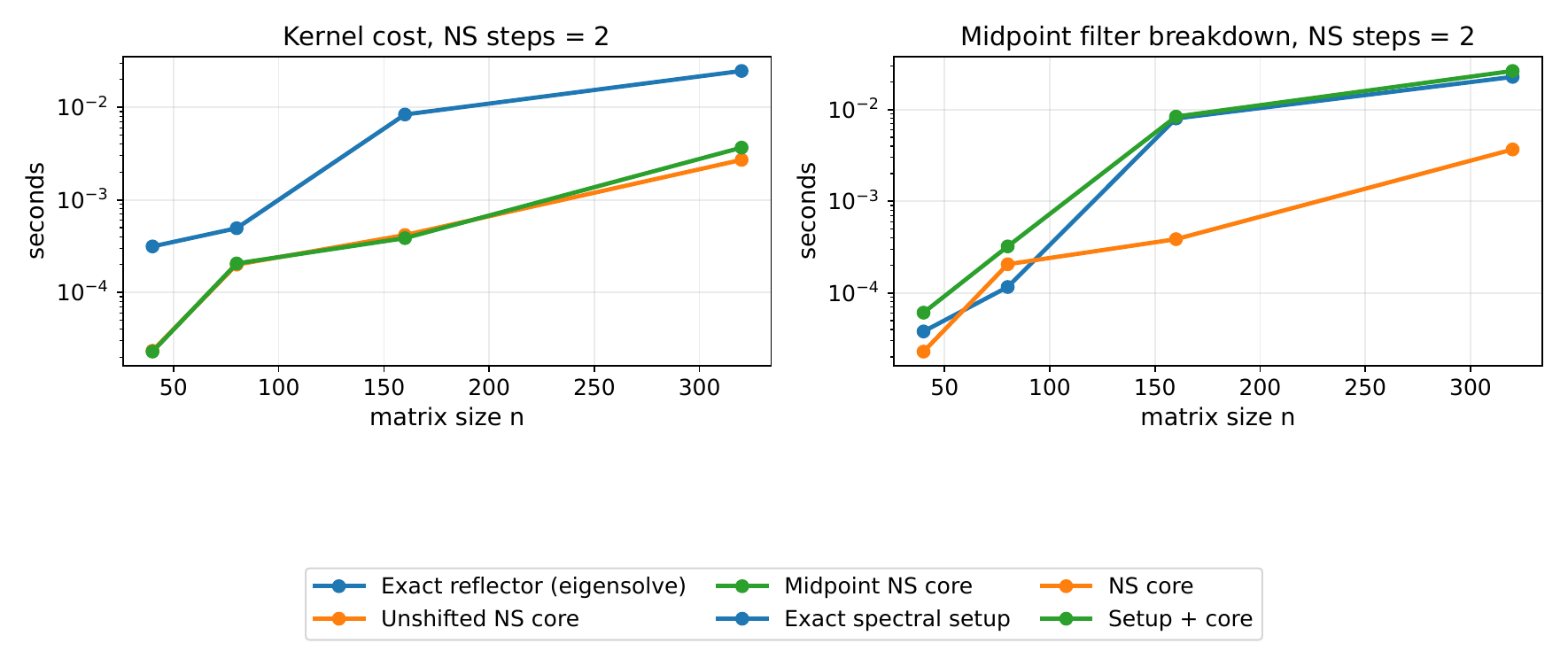}
\caption{Dense explicit-Hessian timing benchmark at Newton--Schulz depth
\(m=2\).  Left: kernel-level comparison between a full-eigensolve exact
reflector and Newton--Schulz sign cores with shift and scale already supplied.
Right: midpoint-filter cost when exact spectral setup is included.}
\label{fig:cost}
\end{figure}

Taken together, the experiments give a regime map.  The shifted sign is the
right matrix-function object for prescribed-index reflection, raw \(\sgn(H)\) is
not index-selective, midpoint shifting is the clean default for finite-step
filters, and Newton--Schulz depth must be chosen in relation to the scaled
spectral margin.  The dense-kernel route is not a replacement for subspace
tracking in all settings; it is a complementary realization for explicit-Hessian
problems where full-matrix kernels, shift reuse, or non-small target indices
make matrix-function primitives attractive.

\section{Conclusion}

This paper studies prescribed-index reflection as a matrix-function problem.
For a symmetric matrix \(H\), the reflector \(I-2P_k\) associated with the first
\(k\) eigenvectors is exactly \(\sgn(H-sI)\) whenever the shift lies in the target
spectral gap.  This identity separates the exact spectral object from the
numerical sign engine used to approximate it.

The main analytical point is that exact sign evaluation is not the minimal local
requirement.  Odd sign-preserving filters inherit the same local
prescribed-index structure, while quantitative accuracy is governed by the
finite-step scalar sign error on the shifted spectrum.  For Newton--Schulz, this
leads to a gap-dependent operator bound and to a simple explanation of spectral
compression under conservative scaling.  The shift analysis shows that the
midpoint maximizes the worst-case separation from the singularity of the sign
function and gives deterministic certificates for inexact and reused shifts.

The experiments support this interpretation.  Controlled spectra verify the
margin predictions, low-dimensional trajectories distinguish shifted signs from
raw signs, target-index scans show how the comparison changes as \(k\) grows,
and Allen--Cahn and timing tests identify the stiff and dense-kernel regimes.
The resulting conclusion is not that full-matrix sign filters universally
replace subspace tracking.  Instead, shifted matrix-sign reflectors provide a
complementary dense explicit-Hessian realization whose effectiveness depends on
the target index, the scaled spectral margin, the desired reflector accuracy, and
the cost of obtaining or reusing shift and scale information.

Future work should develop adaptive-depth and rational sign engines for stiff
spectra, nonasymptotic links between reflector error and outer iteration
complexity, and amortized or warm-started shift policies for slowly varying
Hessian sequences.

\section*{Data and Code Availability}

No external data sets were used in this study.  Code, tabulated outputs, and
figure-generation scripts used for the numerical results are provided in the
supplementary materials associated with this submission.

\bibliographystyle{\siambibliographystyle}
\bibliography{references}

\end{document}